\documentclass[a4paper,11pt]{article}
\usepackage[left = 2cm, right = 2cm, top = 2.4cm, bottom = 2.4cm]{geometry}

\usepackage{amsmath,amsthm,amssymb}
\usepackage{mathtools}

\usepackage[hidelinks]{hyperref}
\usepackage[nameinlink,capitalise,noabbrev]{cleveref}
\usepackage{cite}

\usepackage{graphicx,xcolor}

\usepackage[linesnumbered,ruled,vlined]{algorithm2e}

\usepackage{multirow}

\theoremstyle{definition}
\newtheorem{definition}{Definition}[section]
\newtheorem{remark}[definition]{Remark}

\theoremstyle{plain}
\newtheorem{theorem}[definition]{Theorem}
\newtheorem{proposition}[definition]{Proposition}
\newtheorem{claim}[definition]{Claim}
\newtheorem{question}[definition]{Question}

\def \sm {\setminus}
\def \cl {\colon}
\def \ce {\coloneqq}
\def \F {\mathbb{F}}
\def \P {\mathbb{P}}

\renewcommand{\le}{\leqslant}
\renewcommand{\ge}{\geqslant}
\renewcommand{\leq}{\leqslant}

\def \es {\varnothing}
\renewcommand \b[2] {\binom{#1}{#2}}

\def \mC {\mathcal{C}}
\def \mE{\mathcal{E}}
\def \mG {\mathcal{G}}
\def \mH{\mathcal{H}}
\def \mS {\mathcal{S}}
\def \mT {\mathcal{T}}

\newcommand{\floor}[1]{\left \lfloor #1 \right \rfloor}

\DeclareMathOperator{\poly}{poly}

\usepackage{etoolbox}
\patchcmd{\thebibliography}{\labelsep}{\labelsep\itemsep=2pt \parsep=1.5pt \relax}{}{}

\newcommand{\fses}{f_{\text{sES}}}
\newcommand{\nses}{n_{\text{sES}}}
\newcommand{\fwes}{f_{\text{wES}}}
\newcommand{\nwes}{n_{\text{wES}}}

\title{Property O and Erd\H{o}s--Szekeres properties in linear hypergraphs}

\author{
	Suyun Jiang\thanks{School of Artificial Intelligence, Jianghan University, Wuhan, Hubei, China. Supported by Hubei Provincial Natural Science Foundation of China (No. 2025AFB666), National Natural Science Foundation of China (No. 11901246), China Scholarship Council and Institute for Basic Science (IBS-R029-C4). E-mail: \texttt{jiang.suyun@163.com}. }
	\and Ander Lamaison\thanks{Universidad P\'ublica de Navarra, Edificio Las Encinas, Campus de Arrosad\'{\i}a, 31006 Pamplona, Spain.
     E-mail: \texttt{ander.lamaison@unavarra.es}. }
	\and Minghui Ouyang\thanks{School of Mathematical Sciences, Zhejiang University, Hangzhou 310058, China. E-mail: \texttt{ouyangminghui1998@gmail.com}. }}

\date{\vspace{-5ex}}

\begin{document}
	\maketitle
	\begin{abstract}
		An oriented $k$-uniform hypergraph, or oriented $k$-graph, is said to satisfy \emph{Property~O} if, for every linear ordering of its vertex set, there is some edge oriented consistently with this order. The minimum number $f(k)$ of edges in a $k$-graph with Property~O was first studied by Duffus, Kay, and R\"{o}dl; the bounds were later improved by Kronenberg, Kusch, Lamaison, Micek, and Tran. In particular, they established the bounds $k! + 1 \le f(k) \le \left(\lfloor\tfrac{k}{2}\rfloor+1 \right) k! - \lfloor\tfrac{k}{2}\rfloor(k-1)!$ for every $k \ge 2$.

		In this note, we extend the study of Property~O to the \emph{linear} setting. We determine the minimum number $f'(k)$ of edges in an oriented linear $k$-graph with Property~O up to a $\operatorname{poly}(k)$ multiplicative factor, showing that $\frac{(k!)^2}{2e^2k^4} \le f'(k) \le (1+o(1)) \cdot 4 k^6 \ln^2 k \cdot (k!)^2$. Our approach also yields bounds on the minimum number $n'(k)$ of vertices in an oriented linear $k$-graph with Property~O. Additionally, we explore the minimum number of edges and vertices required in a linear $k$-graph satisfying the newly introduced Erd\H{o}s--Szekeres properties.

	\end{abstract}

\section{Introduction}
	A central type of problem in extremal combinatorics asks: what is the maximum or minimum size of a finite structure that satisfies a given property? A classical example is \emph{Property B}, introduced by Bernstein in 1907~\cite{Ber07}, which concerns the minimum number of edges in a $k$-uniform hypergraph such that every two-coloring of its vertex set contains a monochromatic edge. For further developments, see~\cite{Erd63, Erd64, Bec78, Spe81, RS00, CK15, AS16}.

	In 2018, Duffus, Kay, and R\"{o}dl~\cite{DKR18} introduced an analogue of Property B, called \emph{Property~O}, where coloring is replaced by ordering. An \emph{oriented $k$-uniform hypergraph}, or \emph{oriented $k$-graph}, is a pair $\mH = (V, \mE)$, where $V$ is a finite vertex set and $\mE$ is a family of ordered $k$-tuples of distinct vertices, with the condition that no two $k$-tuples correspond to the same underlying set. If $\mE$ contains one $k$-tuple for every $k$-subset of $V$, then $\mH$ is called a $k$-\emph{tournament}. A $k$-tuple $(x_1, x_2, \ldots, x_{k})$ is said to be \emph{consistent} with a linear order $<$ on $V$ if $x_1 < x_2 < \cdots < x_{k}$.

	\begin{definition}[Property O]
		An oriented $k$-graph $\mH = (V, \mE)$ is said to have \emph{Property~O} if, for every linear order $<$ on $V$, there exists an edge $\vec{e} \in \mE$ consistent with $<$. Let $f(k)$ and $n(k)$ denote the minimum number of edges and vertices, respectively, in an oriented $k$-graph with Property~O. Formally,
        \begin{gather*}
            f(k) \ce \min \Bigl\{ |\mE | \cl \text{there exists an oriented $k$-graph $\mH =(V,\mE)$ with Property O} \Bigr\}, \\
            n(k) \ce \min \Bigl\{ |V| \cl \text{there exists an oriented $k$-graph $\mH =(V,\mE)$ with Property O} \Bigr\}.
        \end{gather*}
	\end{definition}

	It is straightforward to check that $f(2)=3$, for example by considering a cyclically oriented triangle. The study of $f(k)$ and $n(3)$ was initiated by Duffus, Kay, and R\"{o}dl~\cite{DKR18}, who proved that $n(3) \in \{6,7,8,9\}$ and established the following bounds.

	\begin{theorem}[\hspace{-0.005em}\cite{DKR18}] \label{thm:DKR18}
		The function $f(k)$ satisfies
			\[ k! \le f(k) \le \left(k^2\ln k\right) k!, \]
		where the lower bound holds for all $k$, and the upper bound holds for sufficiently large $k$.
	\end{theorem}

	The lower bound in \Cref{thm:DKR18} follows from a simple counting argument. For the upper bound, the authors showed that a random $k$-tournament on
		\[ n=(k/e)^2 \Bigl(\pi \cdot \exp(e^2/2)\cdot k^3 \ln k \Bigr)^{1/k} = \Theta(k^2) \]
    vertices, which has at most $k^2 \ln k \cdot k!$ edges, possesses Property~O with positive probability. They further proved that almost all $k$-tournaments with $(c-o(1)) \sqrt{k} \cdot k!$ edges fail to have Property~O, where $0<c<1$ is an absolute constant.

	Subsequently, Kronenberg, Kusch, Lamaison, Micek, and Tran~\cite{KKLMT19} determined that $n(3)=6$ and sharpened the bounds in \Cref{thm:DKR18}. The improvement of $+1$ in the lower bound arises from an elegant number-theoretic argument showing that it is impossible for each ordering to be consistent with exactly one edge. The upper bound, on the other hand, was achieved via an explicit construction. Notably, the extremal construction is quite sparse: it uses roughly $k!$ vertices, most of which have degree $\floor{k/2} + 1$, in contrast to the random tournament construction in~\cite{DKR18}, which is a complete $k$-graph.

	\begin{theorem}[\hspace{-0.005em}\cite{KKLMT19}]
        For every integer $k \ge 3$, we have
        	\[ k!+1 \le f(k) \le \Bigl(\lfloor \tfrac{k}{2} \rfloor + 1 \Bigr)k! - \lfloor \tfrac{k}{2} \rfloor (k-1)!. \]
    \end{theorem}

    The first aim of this paper is to establish quantitative bounds for oriented linear $k$-graphs with Property~O. A $k$-graph is called \emph{linear} if every two distinct edges intersect in at most one vertex. For convenience, we say that an oriented $k$-graph $\mH$ has the \emph{linear Property~O} if it is linear and satisfies Property~O.

    \begin{definition}[Linear Property O]
        An oriented linear $k$-graph $\mH = (V, \mE)$ is said to have \emph{linear Property~O} if, for every linear order $<$ on $V$, there exists an edge $\vec{e} \in \mE$ consistent with $<$. Let $f'(k)$ and $n'(k)$ denote the minimum number of edges and vertices, respectively, in an oriented linear $k$-graph with linear Property~O. Formally,
        \begin{gather*}
			f'(k) \ce \min \Bigl\{ |\mE | \cl \text{there exists an oriented linear $k$-graph $\mH =(V,\mE)$ with Property O} \Bigr\}, \\
			n'(k) \ce \min \Bigl\{ |V| \cl \text{there exists an oriented linear $k$-graph $\mH =(V,\mE)$ with Property O} \Bigr\}.
        \end{gather*}
	\end{definition}

    Linear hypergraphs are a natural and well-studied restriction in extremal combinatorics: requiring that any two edges share at most one vertex rules out trivial dense constructions and often leads to questions of independent interest. It is therefore natural to ask how Property~O behaves in this setting. Beyond this, linear hypergraphs with Property~O can be used to translate between the vertex-ordered and vertex-unordered versions of certain graph problems. We illustrate this by using the Erd\H{o}s--Hajnal conjecture as an example. This conjecture states that, for every graph $F$, there exists a constant $c>0$ such that every graph $G$ on $n$ vertices not containing $F$ as an induced subgraph has a clique or an independent set of size at least $n^c$.

    Suppose that a counterexample is found to the vertex-ordered version of the Erd\H{o}s--Hajnal conjecture. Specifically, suppose that we find a vertex-ordered graph $F$, and a sequence of vertex-ordered graphs $\{G_n\}_{n=1}^\infty$, where $|V(G_n)|\to\infty$, and its largest clique and independent set have size $|V(G_n)|^{o(1)}$, such that $G_n$ does not contain $F$ as a vertex-ordered induced subgraph. Let $v_1 \prec v_2 \prec \cdots \prec v_r$ be the vertices of $F$, and let $\mH = (V,\mE)$ be an oriented linear $r$-graph with Property~O. Replace each edge $\vec{e} = w_1w_2 \cdots w_r \in \mE$ with a vertex-ordered copy of $F$, and denote the result by $F'$.

    Since $\mH$ is linear, the vertices $\{w_1, w_2, \cdots, w_r\}$ induce a copy of $F$ (no other edge $\vec{e'} \in \mE$ introduces additional edges among them). Moreover, because $\mH$ has Property~O, any linear order $<$ of $V$ is consistent with some $\vec{e} \in \mE$, implying that $F'$ with the linear order $<$ contains $F$ as an ordered induced subgraph. As a consequence, $G_n$ cannot contain $F'$ as an induced subgraph. Therefore, the unordered graph sequence $\{G_n\}_{n=1}^\infty$ together with the graph $F'$ forms a counterexample to the original Erd\H{o}s--Hajnal conjecture.

	We establish the following quantitative bounds for $f'(k)$ and $n'(k)$:

	\begin{theorem} \label{thm:main_linear_prop_O}
        For every integer $k \ge 2$, we have
        \begin{gather*}
			\frac{(k!)^2}{2e^2k^4} \le f'(k) \le (1+o(1)) \cdot 4 k^6 \ln^2 k \cdot (k!)^2, \\
			\frac{(k-1) \cdot (k-1)!}{ek} \le n'(k) \le (1+o(1)) \cdot 2 k^4 \ln k \cdot k!.
        \end{gather*}
        The lower bounds hold for every $k \ge 2$, whereas the upper bounds hold asymptotically as $k \to \infty$.
	\end{theorem}

    To summarize, we determine $f'(k)$ and $n'(k)$ up to a $\poly(k)$ multiplicative factor. In particular, up to polynomial factors in $k$, $f'(k)$ is on the order of $(k!)^2$ and thus roughly the square of $f(k)$, and $n'(k)$ is on the order of $k!$, whereas $n(k)$ is bounded by a polynomial in $k$. A detailed comparison is provided in~\Cref{table:edges} and~\Cref{table:vertices}.

	In addition to Property~O, it is also natural to require that the orders on the edges be mutually consistent and extend to a global linear order $<_1$. However, one cannot use the same definition as in Property~O, since otherwise the reverse $<_2$ of the global order would fail to be consistent with $<_1$ on every edge. To address this, we require that any ordering $<_2$ of $V$ be \emph{monotonically consistent} with some edge, meaning that $<_2$ either induces the same order on that edge as $<_1$, or its reverse. We refer to this as the \emph{Erd\H{o}s--Szekeres Property}, owing to its similarity with the classical Erd\H{o}s--Szekeres theorem.

    \begin{theorem}[Erd\H{o}s--Szekeres~\cite{ES35}]
        For any positive integer $k$, every sequence of distinct real numbers of length at least $(k-1)^2 + 1$ contains a monotonically increasing or decreasing subsequence of length $k$.
    \end{theorem}

    Depending on whether a global order on the vertex set is fixed, we distinguish between a \emph{weak} and a \emph{strong} version, defined as follows.

	\begin{definition}[Weak and strong Erd\H{o}s--Szekeres Property]
        Let $\mH = (V, \mE)$ be a $k$-graph. We say that $\mH$ has the \emph{weak Erd\H{o}s--Szekeres Property} if there exists a linear order $<_1$ on $V$ such that, for every linear order $<_2$ on $V$, there exists an edge $e \in \mE$ on which $<_1$ and $<_2$ are monotonically consistent. We say that $\mH$ has the \emph{strong Erd\H{o}s--Szekeres Property} if the above holds for every choice of $<_1$.

        We denote by $\fwes(k)$ (resp. $\fses(k)$) the minimum number of edges in a $k$-graph with the weak (resp. strong) Erd\H{o}s--Szekeres Property, and by $\nwes(k)$ (resp. $\nses(k)$) the minimum number of vertices. Formally,
        \begin{gather*}
			\fwes(k) \ce \min \Bigl\{ |\mE| \cl \exists \, \mH=(V,\mE) \text{ with weak Erd\H{o}s--Szekeres Property} \Bigr\}, \\
			\fses(k) \ce \min \Bigl\{ |\mE| \cl \exists \, \mH=(V,\mE) \text{ with strong Erd\H{o}s--Szekeres Property} \Bigr\}, \\
			\nwes(k) \ce \min \Bigl\{ |V| \cl \exists \, \mH=(V,\mE) \text{ with weak Erd\H{o}s--Szekeres Property} \Bigr\}, \\
			\nses(k) \ce \min \Bigl\{ |V| \cl \exists \, \mH=(V,\mE) \text{ with strong Erd\H{o}s--Szekeres Property} \Bigr\}.
        \end{gather*}
    \end{definition}

    Similarly to the case of Property~O, linear hypergraphs with the Erd\H{o}s--Szekeres property can be used in constructions in other graph and hypergraph problems. For an example of this, see how the proof of Theorem 1.2 in~\cite{Lam24} uses linear hypergraphs with the weak Erd\H{o}s--Szekeres property (Lemma 6.3 in the same paper).

    Clearly, we have $\fwes(k)\leq \fses(k)$ and $\nwes(k)\leq \nses(k)$. Moreover, the classical Erd\H{o}s--Szekeres theorem implies the bounds
		\[ \fses(k) \,\le\, \b{(k-1)^2+1}{k} \quad \text{and} \quad \nwes(k) \,=\, \nses(k)=(k-1)^2+1. \]

    For $\fwes(k)$, we obtain a slightly improved bound. It is natural to ask whether the bound for $\fses(k)$ can also be improved, that is, whether fewer edges suffice to achieve the strong Erd\H{o}s--Szekeres property; see \Cref{prob:value_for_ses}.

    \begin{theorem} \label{thm:bounds_for_fwes(k)}
		For every $k \ge 3$, we have
			\[ \frac{k!}{2}+1 \le \fwes(k) \le \b{(k-1)^2+1}{k} - \b{(k-1)^2 - 1}{k-1}. \]
	\end{theorem}

	As in the case of linear Property~O, we now consider the Erd\H{o}s--Szekeres property for linear $k$-graphs. Let $\fwes'(k), \fses'(k), \nwes'(k)$, and $\nses'(k)$ denote the minimum number of edges (resp. vertices) in a linear $k$-graph with the weak (resp. strong) Erd\H{o}s--Szekeres property. We obtain the following bounds:

	\begin{theorem} \label{thm:main_linear_ES}
        For every integer $k \ge 2$, we have
        \begin{align*}
			\frac{(k!)^2}{8e^2k^4} \le \fwes'(k) \le \fses'(k) &\le (1+o(1)) \cdot 16 k^8 \ln^2 k \cdot \b{(k-1)^2+1}{k}^2 \\
				&= (1+o(1)) \cdot \frac{8 k^{2k+7} e^{2k} \ln^2 k}{\pi e^{5}},
        \end{align*}
        and
        \begin{align*}
			\frac{(k-1) \cdot (k-1)!}{2ek} \le \nwes'(k) \le \nses'(k)
				&\le (1+o(1)) \cdot 4k^7 \ln k \cdot \b{(k-1)^2+1}{k} \\
            	&= (1+o(1)) \cdot \frac{2\sqrt{2} \cdot k^{k+6.5} \cdot e^k \ln k}{\sqrt{\pi} \cdot e^{2.5}}.
        \end{align*}
        The lower bounds hold for every $k \ge 2$, whereas the upper bounds hold asymptotically as $k \to \infty$.
	\end{theorem}

	\Cref{table:edges} summarizes the known results for Property~O and the Erd\H{o}s--Szekeres properties on the number of edges, while \Cref{table:vertices} does the same for the number of vertices.

    \begin{table}[htbp]
		\caption{Results for $f_.^\prime(k)$ and $f_.(k)$}
        \renewcommand\arraystretch{1.4}
		\begin{center}\scalebox{0.8}{
				\begin{tabular}{c|c|c|c|c}
					\hline
					\multirow{2}{*}{\textbf{Properties}}& \multicolumn{2}{c|}{\textbf{linear $k$-graph}  ($f_.^\prime(k)$)}	&	\multicolumn{2}{|c}{\textbf{$k$-graph} ($f_.(k)$)} \\
                    \cline{2-5}
                     & Lower bound & Upper bound & Lower bound & Upper bound\\
					\hline
					\textbf{Property O}	&$\frac{(k!)^2}{2e^2k^4}$ &  $(1+o(1)) \cdot 4 k^6 \ln^2 k \cdot (k!)^2$	&$k!+1$ \cite{KKLMT19} & $\left(\floor{\frac{k}{2}}+1 \right) k! - \floor{\frac{k}{2}}(k-1)!$ \cite{KKLMT19}	\\
					\hline
					\textbf{weak Erd\H{o}s--}& \multirow{4}{*}{$\frac{(k!)^2}{8e^2k^4}$} & \multirow{4}{*}{$(1+o(1)) \cdot \frac{8 k^{2k+7} e^{2k} \ln^2 k}{\pi e^5}$} &\multirow{4}{*}{$\frac{k!}{2}+1$ for $k \ge 3$} & \multirow{2}{*}{$\b{(k-1)^2+1}{k} - \b{(k-1)^2 - 1}{k-1}$}	\\
					\textbf{Szekeres Property}&&&& \\
                    \cline{1-1}
                    \cline{5-5}
					\textbf{strong Erd\H{o}s--}&						 &  & & \multirow{2}{*}{$\b{(k-1)^2+1}{k}$ \cite{ES35}}	\\
					\textbf{Szekeres Property}&&&& \\
					\hline
			\end{tabular}}
		\end{center}
		\label{table:edges}
	\end{table}
    \begin{table}[htbp]
		\caption{Results for $n_.^\prime(k)$ and $n_.(k)$}
        \renewcommand\arraystretch{1.4}
		\begin{center}\scalebox{0.8}{
				\begin{tabular}{c|c|c|c|c}
					\hline
					\multirow{2}{*}{\textbf{Properties}}& \multicolumn{2}{c|}{\textbf{linear $k$-graph}  ($n_.^\prime(k)$)}	&	\multicolumn{2}{|c}{\textbf{$k$-graph} ($n_.(k)$)} \\
                    \cline{2-5}
                     & Lower bound & Upper bound & Lower bound & Upper bound\\
					\hline
					\multirow{2}{*}{\textbf{Property O}}	& \multirow{2}{*}{$\frac{(k-1) \cdot (k-1)!}{ek}$} &  \multirow{2}{*}{$(1+o(1)) \cdot 2 k^4 \ln k \cdot k!$}	& $(1+o(1)) \left( \frac{k}{e} \right)^2 = (k!)^{2/k}$ \cite{DKR18} & $(1+o(1))\left( \frac{k}{e} \right)^2$ \cite{DKR18}	\\
                    \cline{4-5}
						&  &  	& \multicolumn{2}{|c}{$n(3)=6$ \cite{KKLMT19}}	\\
					\hline
					\textbf{weak (strong) Erd\H{o}s--}& \multirow{2}{*}{$\frac{(k-1) \cdot (k-1)!}{2ek}$} & \multirow{2}{*}{$(1+o(1)) \cdot \frac{2\sqrt{2} \cdot k^{k+6.5} \cdot e^k \ln k}{\sqrt{\pi} \cdot e^{2.5}}$} &\multicolumn{2}{|c}{\multirow{2}{*}{$\nwes(k) = \nses(k) = (k-1)^2+1$ \cite{ES35}}}	\\
					\textbf{Szekeres Property}&& &\multicolumn{2}{|c}{} \\
					\hline
			\end{tabular}}
		\end{center}
        \label{table:vertices}
	\end{table}

    \paragraph{Paper Organization.}
    We study Property~O and the Erd\H{o}s--Szekeres properties for linear hypergraphs, and present the proofs of \Cref{thm:main_linear_prop_O} and \Cref{thm:main_linear_ES} in \Cref{sec:linear_prop_O} and \Cref{sec:linear_ES}, respectively. In \Cref{sec:fwes}, we further study the weak Erd\H{o}s--Szekeres property for general $k$-graphs, without assuming linearity. Finally, \Cref{sec:concl_rmk} concludes with remarks and open problems.

\section{Proof of \texorpdfstring{\Cref{thm:main_linear_prop_O}}{Theorem \ref{thm:main_linear_prop_O}}} \label{sec:linear_prop_O}
	The Lov\'asz Local Lemma is a fundamental result in the probabilistic method. We apply it here to show the existence of a total ordering that is inconsistent with all edges, provided the number of edges is small.

	\begin{proposition}[Lov\'asz Local Lemma \cite{Spe77,She85}]
		Let $\{A_1, A_2, \cdots, A_k\}$ be a collection of events in some probability space. Let $0<p<1$ and $d \in \mathbb{Z}_{\ge0}$ be constants such that $epd < 1$, where $e \approx 2.718$ is the base of the natural logarithm. If each event $A_i$ occurs with probability at most $p$ and these events admit a dependency graph of maximum degree at most $d$, then the probability that none of these events occur is strictly positive; that is, $\P \left( \bigwedge_{i=1}^k \overline{A_{i}} \right) > 0$.
	\end{proposition}

	\begin{proof}[\underline{Lower bound in \Cref{thm:main_linear_prop_O}}]
		It suffices to consider $k \ge 3$, since the case $k = 2$ is trivial. Let $\mH=(V(\mH), \mE(\mH))$ be an oriented linear $k$-graph satisfying Property~O. We aim to prove that $|\mE(\mH)| \ge \frac{(k!)^2}{2e^2k^4}$. Without loss of generality, let $V(\mH) = [n]$, and consider a uniformly random permutation $\sigma$ of $[n]$. For each edge $\vec{e} \in \mE(\mH)$, let $A_{\vec{e}}$ denote the event that $\sigma$ is consistent with $\vec{e}$. Note that the set of events $\bigl\{ A_{\vec{f}} \cl \vec{f} \in \mE(\mH),\, \vec{e} \cap \vec{f} = \es \bigr\}$ depends only on the relative ordering of vertices outside $\vec{e}$. Hence $A_{\vec{e}}$ is independent of this family of events. Moreover, since $\mH$ satisfies Property~O, we have $\P \left( \bigvee_{\vec{e} \in \mE(\mH)} A_{\vec{e}} \right) = 1$.

		Let $\Lambda \ce \floor{ \frac{(k-1)!}{ek} } + 2$, and let $S \ce \{v \in V(\mH) \cl \deg(v) \ge \Lambda \}$ be the set of vertices of degree at least $\Lambda$. For each edge $\vec{e} \in \mE(\mH)$ with $\vec{e} \cap S \ne \es$, fix an arbitrary vertex $v_{\vec{e}} \in \vec{e} \cap S$. We now construct a new oriented hypergraph $\mH'$ with vertex set $V(\mH') \ce V(\mH)$ and edge set
			\[\mE(\mH') \ce \Bigl\{\vec{e} \cl \vec{e} \in \mE(\mH),\, \vec{e} \cap S = \es \Bigr\} \cup \Bigl\{\vec{e} \sm \{v_{\vec{e}}\} \cl \vec{e} \in \mE(\mH),\, \vec{e} \cap S \neq \es \Bigr\}. \]
		Each edge in $\mE(\mH')$ inherits its orientation from the corresponding edge in $\mH$ (either unchanged or with the designated vertex removed).

		By construction, every linear order of $V(\mH')$ is consistent with some edge of $\mH'$. Moreover, since each edge of $\mH'$ arises from an edge of $\mH$ by removing at most one vertex, we have $\deg_{\mH'}(v) \le \deg_{\mH}(v)$ for all vertices $v$. We now analyze the degree of the vertices in $\mH'$.
        \vspace{-1.5ex}
        \begin{itemize}
            \item If $v \notin S$, then $\deg_{\mH'}(v) \le \deg_{\mH}(v) \le \Lambda - 1$.
            \vspace{-1.5ex}
            \item If $v \in S$, let us partition the edges containing $v$ into two sets:
					\[ A \ce \Bigl\{ \vec{e} \in \mE(\mH) \cl v \in \vec{e} ,\, |\vec{e} \cap S| = 1 \Bigr\}, \quad B \ce \Bigl\{ \vec{e} \in \mE(\mH) \cl v \in \vec{e} ,\, |\vec{e} \cap S| \ge 2 \Bigr\}. \]
				Edges in $A$ will have $v$ removed in $\mH'$, so they do not contribute to $\deg_{\mH'}(v)$. For edges in $B$, since $\mH$ is linear, each other vertex $u \in S$ shares at most one edge with $v$, hence $|B| \le |S| - 1$. Thus $\deg_{\mH'}(v) \le |B| \le |S|-1$.
        \end{itemize}
        \vspace{-1.5ex}
        Hence the maximum degree in $\mH'$ is at most $\max\{\Lambda-1, |S|-1\}$.

		Now, for each $\vec{e} \in \mE(\mH')$, define $A'_{\vec{e}}$ to be the event that a random permutation $\sigma$ is consistent with $\vec{e}$. These events also satisfy the independence condition, i.e., $A_{\vec{e}}'$ is independent of the family $\bigl\{ A_{\vec{f}}' \cl \vec{f} \in \mE(\mH') ,\, \vec{e} \cap \vec{f} = \es \bigr\}$, and $\P \left( \bigvee_{\vec{e} \in \mE(\mH')} A_{\vec{e}}^\prime \right) = 1$.

		Suppose $|S| \le \Lambda$. Consider an arbitrary edge $\vec{e} \in \mE(\mH')$. Since $|\vec{e}| \le k$ and the maximum degree of $\mH'$ is at most $\max\{\Lambda-1, |S|-1\} = \Lambda-1$, the number of edges intersecting $\vec{e}$ is at most
			\[ \left| \{\vec{e'} \in \mE(\mH') \cl \vec{e'} \neq \vec{e} \text{ and } \vec{e'} \cap \vec{e} \neq \es \} \right| \,\le\, k \cdot (\Lambda-2) \,=\, k \cdot \floor{ \frac{(k-1)!}{ek} } \,\le\, \floor{ \frac{(k-1)!}{e} }. \]
		As $\floor{\frac{(k-1)!}{e}} < \frac{(k-1)!}{e}$, the Lov\'asz Local Lemma with $p=\frac{1}{(k-1)!}$ gives $\P \left( \bigwedge_{\vec{e} \in \mE(\mH')} \overline{A'_{\vec{e}}} \right) > 0$, contradicting Property~O.

		Therefore $|S| \ge \Lambda + 1$. Let $S' \subseteq S$ be an arbitrary subset of size $\Lambda$. Now we can lower bound $|\mE(\mH)|$. Each vertex $v \in S'$ lies in at least $\Lambda$ edges, giving at least $\Lambda |S'|$ total incidences. Since $\mH$ is linear, each pair of vertices in $S'$ is contained in at most one edge, so the number of overcounts due to shared edges is at most $\b{|S'|}{2}$. It follows that
			\[ \bigl|\mE(\mH)\bigr| \,\ge\, \Lambda |S'| - \b{|S'|}{2} \,= \,\Lambda^2 - \b{\Lambda}{2} \,=\, \frac{\Lambda(\Lambda+1)}{2} \,\ge\, \frac{(k!)^2}{2e^2k^4}. \]

        For the lower bound on $n'(k)$, we observe that there exists some vertex $v \in S$ of degree at least $\Lambda$ (since $S \neq \varnothing$). Each of the edges containing $v$ includes $k-1$ other vertices. By the linearity of $\mH$, all these $(k-1)\Lambda$ vertices must be distinct, so we have
			\[ \bigl|V(\mH)\bigr| \,\ge\, (k-1)\Lambda + 1 \,>\, \frac{(k-1) \cdot (k-1)!}{ek}. \qedhere \]
	\end{proof}

	\begin{proof}[\underline{Upper bound in \Cref{thm:main_linear_prop_O}}]
		The argument follows the strategy of the upper bound proof in~\cite{DKR18}, where a random orientation of a $k$-tournament is shown to satisfy Property~O with positive probability via a union bound over all vertex orderings. Here we replace the complete $k$-graph by a sparse linear $k$-graph coming from the line system of an affine space, which preserves enough edges to apply the same probabilistic argument while ensuring linearity.

        Let $p$ be the smallest prime at least $k$. By the prime number theorem~\cite{Kou19}, $p=(1+o(1))k$. Let $d$ be the least positive integer such that $p^d\ge2k^3\ln k\cdot k!$. Consider the $d$-dimensional affine space $\F_p^d$ over the finite field $\F_p$. Define $\mH$ as the $p$-uniform hypergraph whose vertex set is $\F_p^d$, and whose edges are the lines in $\F_p^d$.

        Each such edge initially contains $p$ vertices; we shrink each edge arbitrarily to size $k$, and then orient each resulting $k$-set independently and uniformly at random. We show that, with positive probability, the resulting oriented $k$-graph has Property~O.

        Fix an arbitrary linear ordering $\sigma$ of the vertex set $\F_p^d$. For each edge $\vec{e}$, the probability that its orientation is consistent with $\sigma$ is precisely $1/k!$. By independence of edge orientations and a union bound over all possible vertex orderings, we obtain
			\[ \P \biggl( \mH \text{ does not satisfy Property O} \biggr) \le (p^d)! \cdot \left( 1 - \frac{1}{k!} \right)^{\b{p^d}{2} / \b{p}{2}}. \]

        Writing $N=p^d$ and $M=\b{N}{2}/\b{p}{2}$, the choice of $d$ gives $\ln N=(1+o(1))k\ln k$ and $M/(Nk!)\ge(2+o(1))k\ln k$. Thus $N!(1-1/k!)^M\le\exp(N\ln N-M/k!)<1$ for sufficiently large $k$. Hence, with positive probability, there exists an orientation of $\mH$ satisfying Property~O. This construction gives the following bounds on the number of vertices and edges:
        \begin{align*}
			\bigl| V(\mH) \bigr| \,\le\, \bigl| \F_p^d \bigr| \,\le\, p \cdot (1+o(1)) \cdot 2 k^3 \ln k \cdot k! \,&\le\, (1+o(1))k \cdot (1+o(1)) \cdot 2 k^3 \ln k \cdot k! \\
				&=\, (1+o(1)) \cdot 2 k^4 \ln k \cdot k!,
        \end{align*}
        and the number of edges is at most
			\[ \bigl| E(\mH) \bigr| \,\le\, \b{p^d}{2} \Big/ \b{p}{2} \,\le\, (1+o(1)) \cdot 4 k^6 \ln^2 k \cdot (k!)^2. \qedhere \]
	\end{proof}

\section{Proof of \texorpdfstring{\Cref{thm:main_linear_ES}}{Theorem \ref{thm:main_linear_ES}}}\label{sec:linear_ES}
	\begin{proof}[\underline{Lower bound in \Cref{thm:main_linear_ES}}]
        The cases $k=2,3$ are trivial, so assume $k\ge4$ and orient the edges according to a linear order witnessing the weak Erd\H{o}s--Szekeres property.
		The argument closely follows that of \Cref{thm:main_linear_prop_O}. For each edge, there are exactly two vertex orderings that are monotonically consistent with the edge. Hence, for a uniformly random permutation $\sigma$ of the vertex set, the probability that the event $A_e$ occurs is $\frac{2}{k!}$, which is twice the corresponding probability in the proof of \Cref{thm:main_linear_prop_O}. Define $\Lambda \ce \floor{ \frac{(k-1)!}{2ek} } + 2$. After the same deletion construction, the edge sizes are $k-1$ or $k$, so the event probabilities are at most $p=2/(k-1)!<1$. Applying the same argument with this probability bound gives $|S|\ge\Lambda+1$. Repeating the analysis yields the desired lower bounds for $\fwes'(k)$ and $\nwes'(k)$.
	\end{proof}

    For the upper bound, we prove that there exists a sufficiently large, dense linear $k$-graph that satisfies the Erd\H{o}s--Szekeres property. Unlike in the proof of \Cref{thm:main_linear_prop_O}, we cannot assign the orderings on the set of edges independently, since they must extend to a global total order. This obstruction prevents a direct adaptation of the previous approach.

    Nevertheless, by the Erd\H{o}s--Szekeres theorem, any two fixed orderings $\sigma$ and $\tau$ agree monotonically on some $k$-subset within every $(k-1)^2+1$ vertices. Hence, the collection of $k$-subsets on which $\sigma$ and $\tau$ are monotonically consistent is fairly dense. A random linear $k$-graph with sufficiently many edges will intersect this collection with high probability.

    We analyze this via an incremental random construction, where the success probability at each step is readily estimated. While our analysis remains somewhat coarse, we expect that considering a random linear $k$-graph of maximum possible density could yield a $\poly(k)$ improvement. This refinement is not essential, however, as the current gap between the upper and lower bounds for the Erd\H{o}s--Szekeres properties is of order $\exp(k)$ (see \Cref{prob:es_poly_factor}).

	\begin{proof}[\underline{Upper bound in \Cref{thm:main_linear_ES}}]
        We use an incremental random construction to generate a hypergraph $\mH$. We aim to show that, for every pair of orderings $(\sigma, \tau)$, the probability that $\sigma$ and $\tau$ are monotonically consistent on no edge of $\mH$ is less than $\frac{1}{(n!)^2}$. The result then immediately follows from the union bound.

		Let $V(\mH) = [n]$, where the value of $n$ will be chosen later, and set $m \ce \frac{\b{n}{k}}{2k^2 \cdot \b{k}{2}\b{n-2}{k-2}}$. Throughout the algorithm, we maintain a set of candidate edges for the final $k$-graph; we write $\mC_i$ for the value of this set at the start of round $i$ (so $\mC_1 = \b{[n]}{k}$, and $\mC_{i+1}$ is obtained from $\mC_i$ by removing the edge picked in round $i$ together with all edges sharing at least two vertices with it). The construction proceeds in rounds: in round $i$ we pick an edge $e$ uniformly at random from $\mC_i$ and add it to $\mE(\mH)$. To ensure that the resulting $k$-graph is linear, we then delete from $\mC_i$ all edges intersecting $e$ in at least two vertices to form $\mC_{i+1}$. We choose $n$ to be a multiple of $k^4(k-1)^2$, so that $m$ is an integer; this does not affect the asymptotic estimates. The precise procedure is given in Algorithm~\ref{alg:lowerbound_linear_ES}.

		\begin{algorithm}[htbp] \label{alg:lowerbound_linear_ES}
			\caption{\,Incremental random construction}
			$V(\mH) \ce [n]$, $\mE(\mH) \ce \es$ \\
			$\mC_1 \gets \b{[n]}{k}$ \\
			\For{$i = 1, \ldots, m$}{
				Pick $e \gets \mC_i$ uniformly at random \\
				$\mE(\mH) \gets \mE(\mH) \cup \{e\}$ \\
				$\mC_{i+1} \gets \mC_i$ \\
				\For{all $e' \in \mC_{i+1}$ such that $|e' \cap e| \ge 2$}{
					$\mC_{i+1} \gets \mC_{i+1} \sm \{e'\}$
				}
			}
			\textbf{Return} $\mH$
		\end{algorithm}

        \begin{claim}
            The algorithm is well-defined; that is, $\mC_i \ne \es$ for every $i = 1, 2, \ldots, m$.
        \end{claim}

		{\renewcommand\qedsymbol{$\blacksquare$}
		\begin{proof}
			For any $k$-subset $e \subseteq [n]$, the number of $k$-subsets $e' \in \mC_i$ with $|e' \cap e| \ge 2$ is at most $\b{k}{2}\b{n-2}{k-2}$. Consequently, at the start of round $i$, the size of $\mC_i$ is at least
				\[ \bigl| \mC_i \bigr| \,\ge\, \b{n}{k} - (i-1) \cdot \b{k}{2}\b{n-2}{k-2} \,\ge\, \b{n}{k} - \frac{\b{n}{k}}{2k^2 \cdot \b{k}{2}\b{n-2}{k-2}} \cdot \b{k}{2}\b{n-2}{k-2} \,>\, 0, \]
            which proves the claim.
		\end{proof}}

        \begin{claim} \label{clm:2}
            Fix a pair of permutations $(\sigma, \tau)$ of $[n]$. For each round $i = 1,2,\cdots, m$, let $e$ denote the edge chosen in that round. Conditional on any possible history of the first $i-1$ rounds, the probability that $\sigma$ and $\tau$ are monotonically consistent on $e$ satisfies:
				\[ \P \bigg( (\sigma, \tau) \text{ are monotonically consistent on } e \bigg) \,\ge\, \frac{1}{2\b{(k-1)^2+1}{k}}. \]
		\end{claim}

        {\renewcommand\qedsymbol{$\blacksquare$}
        \begin{proof}
            Consider a bipartite graph $\mG$ with parts $[\mS, \mT]$, where $\mS = \b{[n]}{(k-1)^2+1}$ and $\mT = \b{[n]}{k}$. We connect $S \in \b{[n]}{(k-1)^2+1}$ with $T \in \b{[n]}{k}$ if and only if $T \subseteq S$. Thus, every vertex in $\mT$ has degree $\b{n-k}{(k-1)^2+1-k}$. Let $\mG_{\sigma,\tau}$ be the induced subgraph of $\mG$ on $[\mS, \mT']$, where $\mT'$ consists of the $k$-subsets on which $\sigma$ and $\tau$ are monotonically consistent. By the Erd\H{o}s--Szekeres theorem~\cite{ES35}\footnote{Applied to the sequence obtained by listing the elements of $S$ in $\sigma$-order and reading off their positions under $\tau$, which is a sequence of $(k-1)^2+1$ distinct real numbers.}, each $S \in \mS$ is adjacent to at least one vertex in $\mT'$.

            Fix $i \in [m]$ and condition on any possible history $e_1, e_2, \cdots, e_{i-1}$ of the first $i-1$ rounds. Define
            \begin{gather*}
				\mS_i \ce \Bigl\{ S \in \mS \cl |S \cap e_j| \le 1 \text{ for all } j \in [i-1] \Bigr\}, \\
				\mT_i \ce \Bigl\{ T \in \mT' \cl |T \cap e_j| \le 1 \text{ for all } j \in [i-1] \Bigr\}.
            \end{gather*}

            Note that $(\sigma,\tau)$ are monotonically consistent on every $e \in \mT_i$, and $\mT_i \subseteq \mC_i$. Moreover, since the total number of choices for $e$ is at most $|\mC_i| \le \b{n}{k}$, the probability that $(\sigma, \tau)$ are monotonically consistent on $e$ is at least $|\mT_i|/\b{n}{k}$. Thus, we only need to show that $|\mT_i| \ge \b{n}{k} \big/ 2\b{(k-1)^2+1}{k}$.

            Since each $e_j$ intersects with at most $\b{k}{2} \b{n-2}{(k-1)^2-1}$ of the $((k-1)^2+1)$-subsets of $[n]$ in at least $2$ elements, we have
				\[ \Bigl| \bigl\{ S \in \mS \cl |S \cap e_j| \ge 2 \text{ for some } j\in [i-1] \bigr\} \Bigr| \,<\, m \cdot \b{k}{2} \b{n-2}{(k-1)^2-1} \le \frac{1}{2} \cdot \b{n}{(k-1)^2+1}. \]

			Thus,
				\[ \bigl| \mS_i \bigr| \,\ge\, \b{n}{(k-1)^2+1} - \frac{1}{2} \cdot \b{n}{(k-1)^2+1} \,\ge\, \frac{1}{2} \cdot \b{n}{(k-1)^2+1}. \]

            Therefore, by the definitions of $\mS_i$ and $\mT_i$, and writing $N_{\mG_{\sigma,\tau}}(\mS_i)$ for the set of vertices in $\mG_{\sigma,\tau}$ adjacent to some element of $\mS_i$, we have
				\[ \bigl| \mT_i \bigr| \,\ge\, \bigl| N_{\mG_{\sigma,\tau}}(\mS_i) \bigr| \,\ge\, \frac{|\mS_i|}{\max_{T \in \mT}\deg(T)} \,\ge\, \frac{|\mS_i|}{\b{n-k}{(k-1)^2+1-k}}. \]

            This gives
				\[ \bigl| \mT_i \bigr| \,\ge\, \frac{1}{2} \cdot \b{n}{(k-1)^2+1} \cdot \frac{1}{\b{n-k}{(k-1)^2+1-k}} \,=\, \frac{\b{n}{k}}{2 \b{(k-1)^2+1}{k}}. \]

            This completes the proof of the claim.
		\end{proof}}

		By \cref{clm:2} and iterated conditioning, for any pair $(\sigma, \tau)$, we have
			\[ \P \bigg( (\sigma, \tau) \text{ are monotonically consistent on no edge of } \mH \bigg) \,\le\, \left( 1 - \frac{1}{2\b{(k-1)^2+1}{k}}\right)^m. \]

		Thus, we only need $\left( 1 - \frac{1}{2\b{(k-1)^2+1}{k}}\right)^m < \frac{1}{(n!)^2}$ to apply the union bound, which is satisfied if we take
			\[ n \ce (1+o_k(1)) \cdot 4k^7 \ln k \cdot \b{(k-1)^2+1}{k} \,=\, \frac{(1+o_k(1)) \cdot 2\sqrt{2} \cdot k^{k+6.5} \cdot e^k \ln k}{\sqrt{\pi} \cdot e^{2.5}}. \]
        In this case, we obtain
			\[ m \,=\, \frac{\b{n}{k}}{2k^2 \cdot \b{k}{2}\b{n-2}{k-2}} \,=\, (1+o_k(1)) \cdot 16k^8 \ln^2 k \cdot \b{(k-1)^2+1}{k}^2 \,=\, \frac{(1+o_k(1)) \cdot 8 k^{2k+7} \cdot e^{2k} \ln^2 k}{\pi e^{5}}. \qedhere \]
	\end{proof}

    \begin{remark}
        For the weak Erd\H{o}s--Szekeres property, the total number of events $(\sigma, \tau)$ reduces to $n!$. This yields an improved bound for $\nwes'(k)$ by a factor of $2$ and for $\fwes'(k)$ by a factor of $4$. Nevertheless, as mentioned earlier, improving constant factors---or even polynomial factors in $k$---might be achieved by adjusting the parameters in the above proof. Hence we do not attempt such optimizations here.
    \end{remark}

\section{Proof of \texorpdfstring{\Cref{thm:bounds_for_fwes(k)}}{Theorem \ref{thm:bounds_for_fwes(k)}}}\label{sec:fwes}
    Recall that our goal is to show
		\[ \frac{k!}{2}+1 \,\le\, \fwes(k) \,\le\, \b{(k-1)^2+1}{k} - \b{(k-1)^2 - 1}{k-1}, \]
    for all $k \ge 3$.

	The lower bound follows directly from a simple counting argument: a random permutation $<_2$ is monotonically consistent with $<_1$ on any edge with probability $\frac{2}{k!}$. Moreover, since $<_1$ is consistent with itself on every edge, these events cannot be disjoint.

	For the upper bound, set $n \ce (k-1)^2 + 1$, let $V \ce [n]$, and consider the natural order $1 <_1 2 <_1 \cdots <_1 n$ on $[n]$. Define
		\[ \mE \ce \b{[n-1]}{k} \cup \biggl\{ T \cup \{n-1,n\} \cl T \in \b{[n-2]}{k-2} \biggr\}, \]
	the family of $k$-subsets that either avoid $n$ or contain $\{n-1, n\}$. Then we have
		\[ \bigl| \mE \bigr| \,=\, \b{(k-1)^2}{k} + \b{(k-1)^2 - 1}{k-2} \,=\, \b{(k-1)^2+1}{k} - \b{(k-1)^2 - 1}{k-1}. \]

	We now show that $(V, \mE)$ satisfies the weak Erd\H{o}s--Szekeres property.

	Consider an arbitrary linear order $<_2$ on $V$. For each $i \in [n]$, define $a_i$ as the number of vertices less than or equal to $i$ under $<_2$. Equivalently, $a_i = j$ if and only if $i$ is the $j$-th smallest element under $<_2$. Thus $(a_1, \cdots, a_n)$ is a permutation of $[n]$. For each $i = 1,2, \cdots, n$, let $x_i$ and $y_i$ denote the lengths of the longest increasing and decreasing subsequences of $(a_1, \cdots, a_n)$ ending at $a_i$, respectively.

    Following the classical proof of the Erd\H{o}s--Szekeres theorem, the pairs $(x_1, y_1), \cdots, (x_{n-1}, y_{n-1})$ are all distinct: for any $i < j$, at least one of the longest increasing or longest decreasing subsequences ending at $a_i$ can be extended by appending $a_j$.

	If some pair $(x_i, y_i)$, with $i \in [n-1]$, has a coordinate at least $k$, then there exists a monotone subsequence of length $k$ within $(a_1, \cdots, a_{n-1})$. Since $\mE$ contains all $k$-subsets of $[n-1]$, this monotone subsequence yields an edge on which $<_1$ and $<_2$ are monotonically consistent, thus establishing the weak Erd\H{o}s--Szekeres property in this case.

    On the other hand, if all pairs $(x_1, y_1), \cdots, (x_{n-1}, y_{n-1})$ have both coordinates strictly less than $k$, then every ordered pair in $[k-1] \times [k-1]$ appears exactly once among them (since $n = (k-1)^2+1$). In particular, there is some $(x_i, y_i) = (k-1, k-1)$. If $i \neq n-1$, then appending $a_{n-1}$ to either the longest increasing or decreasing subsequence ending at $a_i$ creates a monotone subsequence of length $k$, a contradiction. Thus, we must have $(x_{n-1}, y_{n-1}) = (k-1, k-1)$. In this case, appending $a_n$ to a suitable increasing or decreasing subsequence of length $k-1$ ending at $a_{n-1}$ produces a subsequence of length $k$. Since every $k$-subset containing both $n-1$ and $n$ is in $\mE$, the weak Erd\H{o}s--Szekeres property also holds here. This completes the proof.

\section{Concluding Remarks} \label{sec:concl_rmk}
	In this note, we studied Property~O and the Erd\H{o}s--Szekeres properties for linear $k$-graphs. For Property~O, although we did not consider arbitrary $k$-graphs, this remains an intriguing problem. We believe the following question has an affirmative answer, but at present we are unable to prove even the weaker statement that there exists an absolute constant $c > 1$ such that $f(k) > ck!$ for all sufficiently large $k$.

	\begin{question}[Duffus--Kay--R\"{o}dl \cite{DKR18}]
		Is it true that $\frac{f(k)}{k!} \to \infty$ as $k \to \infty$?
	\end{question}

	The Erd\H{o}s--Szekeres property requires that the local orderings on each edge extend to a global ordering, while allowing either monotone direction. Moreover, the strong version imposes even stricter constraints than the weak one. Thus, any separation result between these three notions (either in the linear or general setting) would be of interest.

	\begin{question}
        Can we establish any separation results between Property~O and the Erd\H{o}s--Szekeres properties, either in the linear or general setting? In particular, is it true that
            \[ \lim\limits_{k \to \infty} \frac{\fwes^\cdot(k)}{f^\cdot(k)} = \infty \quad \text{and} \quad \lim\limits_{k \to \infty} \frac{\fses^\cdot(k)}{\fwes^\cdot(k)} = \infty \,, \]
        where $f^\cdot$ denotes either $f$ or $f'$?
    \end{question}

	Furthermore, for the linear Erd\H{o}s--Szekeres properties, the gap between the best-known upper and lower bounds is $\exp(O(k))$, which is larger than the corresponding gap for Property~O and linear Property~O. This motivates the following problem:

	\begin{question} \label{prob:es_poly_factor}
        Determine $\fwes'(k), \nwes'(k), \fses'(k)$ and $\nses'(k)$ up to a $\poly(k)$ multiplicative factor.
	\end{question}

	Finally, regarding the Erd\H{o}s--Szekeres properties in general $k$-graphs, it is known that with $(k-1)^2+1$ vertices, all edges are required to guarantee the strong Erd\H{o}s--Szekeres property, since there exists a permutation of $\{1, \cdots, (k-1)^2+1\}$ containing only one monotone subsequence of length $k$. A natural question is whether one can achieve the strong Erd\H{o}s--Szekeres property using more vertices but fewer edges.

	\begin{question} \label{prob:value_for_ses}
		Is it true that $\fses(k) = \b{(k-1)^2+1}{k}$?
	\end{question}

\section*{Acknowledgements}
    This work was initiated during the $2^{\text{nd}}$ ECOPRO Student Research Program at the Institute for Basic Science (IBS) in the summer of 2024. MO gratefully acknowledges Hong Liu for hosting his visit to IBS as a student researcher.

\bibliographystyle{plain}
{\small \bibliography{reference}}

\end{document}